\documentclass{article}
\usepackage{amsmath}
\usepackage{amssymb}
\usepackage{amsthm}

\newtheorem{theorem}{Theorem}
\newtheorem*{corollary}{Corollary}

\begin{document}
\title{On an exotic topology of the integers}
\author{Rezs\H o L.~Lovas\thanks{Supported by National Science Research Foundation OTKA No.\ NK 81402.}\and Istv\'an Mez\H o\thanks{Supported by National Science Research Foundation OTKA No.\ NK 79129.}}
\date{2010}
\maketitle

\begin{abstract}
We study some interesting properties of F\"urstenberg's topology of the integers. We show that it is metrizable, totally disconnected, and $(\mathbb Z,+,\cdot)$ is a topological ring with respect to this topology. As an application, we show that any two disjoint sets of primes can be separated by arithmetic progressions.
\end{abstract}

\section{Introduction}

It has been known since Euclid that there are infinitely many prime numbers. The simplest proof runs as follows: if there were a finite number of primes, say $p_1,\dots,p_k$, then $p_1p_2\cdots p_k+1$ would have at least one prime divisor different from each of $p_1,\dots,p_k$. Since then, many other proofs of this fact and much stronger results on the distribution of primes have been found. See Part 1 of \cite{AZ} for a delightful account of some highlights of this history.

In 1955, H.~F\"urstenberg found a topological proof for the fact that there are infinitely many primes \cite{F}. He defined a strange topology on the set $\mathbb Z$ of integers, and, assuming that there is only a finite number of primes, he came to a contradiction. In this topology a set is open if and only if, roughly speaking, each of its points is contained together with an infinite arithmetic progression; we present a precise definition in the next section. This topology is, of course, different from the ordinary topology of $\mathbb Z$, i.e., the subspace topology inherited from $\mathbb R$. The latter is rather bland: every subset of $\mathbb Z$ is open.

Mathematicians usually think that number theory and topology are completely disjoint areas of mathematics. The construction outlined in the previous paragraph shows very vividly that this is not so: this topology establishes a connection between these seemingly so distant branches. That is why we found it so attractive and why we wished to study it in detail. For F\"urstenberg's topology, we asked and answered some of the most natural questions. We also found an application in number theory.

In section 2 we briefly recall F\"urstenberg's proof, and we study some further interesting properties of his topology. Among others we show that it is metrizable, and we explicitly present a translation invariant distance function inducing it. In a metrizable space every two disjoint closed sets can be separated by disjoint open sets containing them. This also has some interesting consequences, which we present in section 3: any two disjoint sets of primes can be separated by arithmetic progressions in a certain sense, which we describe precisely in theorem 4. To our knowledge, the result described in that theorem is completely new.

\section{Description of the topology and its properties}

We denote the set of natural numbers including $0$ by $\mathbb N$, and we let $\mathbb N^*:=\mathbb N\setminus\{0\}$. Similarly, $\mathbb Z^*:=\mathbb Z\setminus\{0\}$.

For $a\in\mathbb Z$ and $b\in\mathbb N^*$ we define
\[
a+b\mathbb Z:=\{a+bn:n\in\mathbb Z\}.
\]
If $a=0$, we simply write $b\mathbb Z$ instead of $0+b\mathbb Z$.
We say that a set $A\subset\mathbb Z$ is \emph{open} if, for each point $a\in A$, there is a number $b\in\mathbb N^*$ such that $a+b\mathbb Z\subset A$. In other words, a subset of $\mathbb Z$ is open if each of its points is contained in an arithmetic progression belonging to the set. Let $\mathcal T$ be the set of all open sets in this sense. It is easy to see that $\mathcal T$ satisfies the usual axioms for a topology, thus $(\mathbb Z,\mathcal T)$ becomes a topological space. So the arithmetic progressions form a basis for the topology. Obviously, the basic sets $a+b\mathbb Z$ are open. More surprisingly, they are also closed, since the complement of $a+b\mathbb Z$ is the union of other arithmetic progressions with the same difference.

How can we use this topology to show that there is an infinite number of primes? Suppose, indirectly, that there is only a finite number of primes, and denote these by $p_1,\dots,p_k$. Then the set
\[
C:=\bigcup_{i=1}^kp_k\mathbb Z
\]
is closed, as it is a \emph{finite} union of closed sets. On the other hand, all integers but $-1$ and $1$ have at least one prime divisor, therefore $C=\mathbb Z\setminus\{-1,1\}$. Thus its complement $\{-1,1\}$ is open, which is clearly a contradiction, and this proves that there are indeed infinitely many prime numbers.

Now we turn to the metrizability property of the topology $\mathcal T$.
If $n\in\mathbb Z^*$, then we define
\[
\|n\|:=\frac 1{\max\{k\in\mathbb N^*:\,1\mid n,\dots,k\mid n\}},
\]
i.e., then ``norm'' $\|n\|$ is the reciprocal of the greatest natural number $k$ with the property that the natural numbers $1,\dots,k$ are all divisors of $n$. Thus, for example,
\[
\|1\|=1,\ \|2\|=\frac 12,\ \|3\|=1,\ \|4\|=\frac 12,\ \|5\|=1,
\ \|6\|=\frac 13,\dots,\|n!\|\le\frac 1n,
\]
and so on. Then, in particular, $\|-n\|=\|n\|$. Furthermore, we set $\|0\|:=0$. Then we define the distance of the integers $m$ and $n$ by $d(m,n):=\|m-n\|$.
\begin{theorem}
With this distance function, $(\mathbb Z,d)$ becomes a metric space, and the metric $d$ induces the topology $\mathcal T$.
\end{theorem}
\begin{proof}
All axioms of a metric space are trivially satisfied except the triangle inequality.
To prove that, first we show that $\|m+n\|\le\|m\|+\|n\|$ if $m,n\in\mathbb Z$. It is sufficient to show this when both $m$ and $n$ are different from $0$.

First suppose $\|m\|\le\|n\|$. Then the numbers $1,2,\dots,\frac 1{\|n\|}$ are all divisors of both $m$ and $n$, thus they are divisors of $m+n$ as well, therefore
\[
\|m+n\|\le\|n\|<\|m\|+\|n\|.
\]
For $\|m\|>\|n\|$ we get the assertion by interchanging the role of $m$ and $n$. From this, the triangle inequality follows easily:
\[
d(m,n)=\|m-n\|=\|m-l+l-n\|\le\|m-l\|+\|l-n\|=d(m,l)+d(l,n)
\]
for any $l,m,n\in\mathbb Z$.

Now we show that the metric $d$ induces $\mathcal T$, indeed. Let $A\subset\mathbb Z$ be an open set with respect to the metric $d$. If $a\in A$, then there is a positive number $r$ such that the open ball $B(a,r)$ with center $a$ and radius $r$ is contained in $A$. Let $b\in\mathbb N^*$ be such that $\|b\|<r$. (Such a $b$ exists: e.g., if $n$ is such that $1/n<r$, then take $b:=n!$.) If $a+bn$ is an arbitrary element of $a+b\mathbb Z$, then
\[
d(a,a+bn)=\|a-(a+bn)\|=\|bn\|\le\|b\|<r,
\]
since the divisors of $b$ are divisors of $bn$ as well. This means that
$a+bn\in B(a,r)$ and $a+b\mathbb Z\subset B(a,r)\subset A$, thus $A$ is open also with respect to the topology $\mathcal T$.

To see the converse, let $A\in\mathcal T$. Thus, if $a\in A$, then $a+b\mathbb Z\subset A$ with some $b\in\mathbb N^*$. Let $r:=1/b$. If $c\in B(a,r)$, i.e., $\|a-c\|<r$, then $b\mid a-c$, thus $c=a+bn$ for some $n\in\mathbb Z$, thus $c\in a+b\mathbb Z$. We have obtained $B(a,r)\subset a+b\mathbb Z\subset A$, i.e., $A$ is open with respect to $d$, too.

According to the previous two paragraphs, $A\in\mathcal T$ if and only if $A$ is open with respect to the metric $d$, which means exactly that $d$ induces the topology $\mathcal T$.
\end{proof}

\begin{corollary}
A sequence $(a_n)_{n\in\mathbb N}$ in $\mathbb Z$ converges to $0$ in the topology $\mathcal T$ if and only if, for every $k\in\mathbb N^*$, there is a number $N\in\mathbb N$ such that $n\ge N$ implies $k\mid a_n$.
\end{corollary}
\begin{proof}
The statement follows from the fact that each neighborhood of $0$ contains a set of the form $k\mathbb Z$.
\end{proof}

Thus a sequence converges to $0$ if and only if any positive $k$ is a divisor of all members whose index is sufficiently large. A typical sequence converging to $0$ in our topology is $(n!)_{n\in\mathbb N}$. This too shows that our topology is different from the ordinary topology of $\mathbb Z$, since in the latter one a sequence can converge to $0$ only if all members but finitely many are zero.

This has a strange consequence as well. Namely, let
\[
a_0:=1,
\quad\mbox{and}\quad a_n:=(n+1)!-n!=n\cdot n!\quad\mbox{if}\quad n\ge 1.
\]
Then
$\sum_{n=0}^\infty a_n=0$
with respect to the topology $\mathcal T$, although every member of this series is a positive integer.

Recall that a topological space is said to be \emph{totally disconnected} if each of its maximal connected subsets consists of one single point.

\begin{theorem}
The topological space $(\mathbb Z,\mathcal T)$ is totally disconnected.
\end{theorem}
\begin{proof}
We show that if a set $A\subset\mathbb Z$ contains at least two different elements $a$ and $b$, then it cannot be connected. Let $k$ be a nonzero integer which is not a divisor of $b-a$. Then
\[
A\cap\{a+nk:n\in\mathbb Z\}\quad\mbox{ and }\quad A\cap\{a+1+nk,\dots,a+k-1+nk:n\in\mathbb Z\}
\]
are nonempty disjoint open subsets of $A$ whose union is $A$, thus $A$ is not connected.
\end{proof}

A \emph{topological ring} is a ring which is a topological space at the same time such that the ring operations are continuous with respect to the product topology, and so is the additive inversion.

\begin{theorem}
The set $\mathbb Z$ is a topological ring with respect to the usual addition and multiplication and the topology $\mathcal T$.
\end{theorem}
\begin{proof}
The additive inversion $n\in\mathbb Z\mapsto -n$ is obviously continuous. To prove the continuity of addition, by translation invariance, it is enough to show that it is continuous at the point $(0,0)$. If $\varepsilon>0$, and $\|a\|,\|b\|<\varepsilon$, then $\|a+b\|<\varepsilon$ by the properties of the norm.

To prove the continuity of multiplication, we show that $a_nb_n\rightarrow ab$ if $a_n\rightarrow a$ and $b_n\rightarrow b$. Indeed, if $k\in\mathbb N^*$, then there exist $N_1,N_2\in\mathbb N$ such that $n\ge N_1$ implies $k\mid a_n-a$, and $n\ge N_2$ implies $k\mid b_n-b$. If $N:=\max\{N_1,N_2\}$, and $n\ge N$, then
\begin{equation*}
a_n\equiv a,\ b_n\equiv b\pmod{k},\quad\mbox{therefore}\quad a_nb_n\equiv ab\pmod{k},
\end{equation*}
i.e., $k\mid a_nb_n-ab$.
\end{proof}

\section{An application}

As an application, we show that any two (possibly infinite) disjoint sets of primes can be separated by arithmetic progressions. So, for example, let
\[A=\{p_1,p_3,\dots\}\quad\mbox{and}\quad B=\{p_2,p_4,\dots\},\]
where $p_i$ is the $i$th positive prime number.

Then there are two sets, say $\mathcal A$ and $\mathcal B$, obtained as unions of arithmetic progressions, such that
\[A\subset\mathcal A,\;B\subset\mathcal B\quad\mbox{and}\quad\mathcal A\cap\mathcal B=\emptyset.\]
More precisely, the next theorem is valid.
\begin{theorem}
Let
\[
A=\{p_0,p_1,p_2,\dots\}\quad
\mbox{and}\quad B=\{q_0,q_1,q_2,\dots\}
\]
two disjoint sets of positive primes. Then there are numbers $a_0,a_1,a_2,\ldots\in\mathbb N^*$ and $b_0,b_1,b_2,\ldots\in\mathbb N^*$ such that
\begin{gather*}
A\subset\{p_0+k_0a_0,p_1+k_1a_1,p_2+k_2a_2,\ldots:k_i\in\mathbb Z\},\\
B\subset\{q_0+k_0b_0,q_1+k_1b_1,q_2+k_2b_2,\ldots:k_i\in\mathbb Z\},
\end{gather*}
and
\[
p_i+k_ia_i\neq q_j+k_jb_j
\]
for any $i,j\in\mathbb N$ and $k_i,k_j\in\mathbb Z$.
\end{theorem}
We remark that the crucial point in the proof is a theorem of Urysohn \mbox{\cite[2.5.7]{KF}:} a topological space with a countable basis is metrizable if and only if the singletons are closed and any two disjoint closed sets can be separated by disjoint open sets.
\begin{proof}
Consider the set
\[
\widetilde{\mathbb Z}:=\{n\in\mathbb Z:n\ge 2\}
\]
with the topology induced by $\mathcal T$. The sets $A$ and $B$ are closed subsets of $\widetilde{\mathbb Z}$, since, e.g.,
\[
\widetilde{\mathbb Z}\setminus A
=\widetilde{\mathbb Z}\cap
\bigcup_{n\in\widetilde{\mathbb Z}\setminus A}n\mathbb Z,
\]
and similarly for $b$. Since $\widetilde{\mathbb Z}$ is a metrizable topological space, two disjoint closed sets can be separated by disjoint open sets, which we may assume to be of the forms $\widetilde{\mathbb Z}\cap U$ and $\widetilde{\mathbb Z}\cap V$, where
\begin{gather*}
U=\{p_0+k_0a_0,p_1+k_1a_1,p_2+k_2a_2,\ldots:k_i\in\mathbb Z\},\\
V=\{q_0+k_0b_0,q_1+k_1b_1,q_2+k_2b_2,\ldots:k_i\in\mathbb Z\}.
\end{gather*}
Note that, a priori, we do not know that $U$ and $V$ are disjoint, we only know that $\widetilde{\mathbb Z}\cap U\cap V=\emptyset$. If, however, we had
\[
p_i+k_ia_i=q_j+k_jb_j
\]
for some $i,j\in\mathbb N$ and $k_i,k_j\in\mathbb Z$, then this Diophantine equation would have infinitely many solutions for $k_i$ and $k_j$, and substituting some of these into the equation would render both sides $\ge 2$, which would contradict $\widetilde{\mathbb Z}\cap U\cap V=\emptyset$. This completes the proof that $U\cap V=\emptyset$.
\end{proof}

It is worth noting that the Diophantine equation
\[
a_ik_i-b_jk_j=q_j-p_i
\]
has a solution if and only if $(a_i,b_j)$, the greatest common divisor of $a_i$ and $b_j$, is a divisor of $p_i-q_j$. Therefore, the condition obtained on $a_i$ and $b_j$ is equivalent to $(a_i,b_j)\nmid\ p_i-q_j$ for any $i,j\in\mathbb N$.

\section*{Note}

Having completed this paper, we realized that in November 2009 another 
metric inducing this topology had been proposed, as far as we make it out, 
by Jim Ferry \cite{forum}. In a way which is
similar to ours, he defines a kind of norm
\[
\|n\|_1:=\sum_{k\in\mathbb N^*,k\nmid n}2^{-k}
\]
for $n\in\mathbb Z$, and then he lets $d_1(m,n):=\|m-n\|_1$ for $m,n\in\mathbb
Z$. He proves that this is indeed a metric, but he leaves it to 
heuristic judgement whether it induces F\"urstenberg's topology.

\end{document}